\documentclass[12pt,oneside]{amsart}
\usepackage{amssymb,verbatim,enumerate,amsmath,ifthen}
\usepackage[mathscr]{eucal}
\usepackage[utf8]{inputenc}
\usepackage[T1]{fontenc}
\usepackage{esint}
\usepackage[marginparwidth=2.4cm]{geometry}

\newtheorem{thm}{Theorem}[section]

\newtheorem{lem}[thm]{Lemma}

\theoremstyle{definition}
\newtheorem{defin}[thm]{Definition}
\newtheorem{rem}[thm]{Remark}

\numberwithin{equation}{section}
\def\eq#1{{\rm(\ref{#1})}}
\def\Eq#1#2{\ifthenelse{\equal{#1}{*}}
  {\begin{equation*}\begin{aligned}[]#2\end{aligned}\end{equation*}}
  {\begin{equation}\begin{aligned}[]\label{#1}#2\end{aligned}\end{equation}}}

\def\D{\mathscr{D}}
\def\E{\mathscr{E}}

\def\M{\mathscr{M}}
\def\P{\mathscr{P}}
\def\F{\mathscr{F}}
\newcommand{\operator}[1]{\mathop{\vphantom{\sum}\mathchoice
{\vcenter{\hbox{\LARGE $#1$}}}
{\vcenter{\hbox{\large $#1$}}}{#1}{#1}}\displaylimits}

\def\Mst_#1^#2{\operator{\mathscr{M}_{\mbox{\scriptsize$\#$}}\!\!}_{#1}^{#2}\,\,}

\newcommand\R{\mathbb{R}}

\newcommand\N{\mathbb{N}}

\DeclareMathOperator{\sign}{sign}

\newcommand{\Hc}{\mathscr{H}}

\DeclareMathOperator{\conc}{conc}
\def\comment#1{}

\title[Estimating the Hardy constant of means]{Estimating the Hardy constant of nonconcave homogenus quasideviation means}

\author{Zsolt P\'ales}
\address{Institute of Mathematics, University of Debrecen, Pf.\ 400, 4002 Debrecen, Hungary}
\email{pales@science.unideb.hu}

\author{Pawe\l{} Pasteczka}
\address{Institute of Mathematics, Pedagogical University of Krak\'ow,  Podchor\k{a}\.{z}ych str 2, 30-084 Krak\'ow, Poland}
\email{pawel.pasteczka@up.krakow.pl}

\thanks{The first author was supported by the K-134191 NKFIH Grant.}


\newcommand\FAM\Phi
\usepackage{color}

\begin{document}
\begin{abstract}
In this paper, we consider homogeneous quasideviation means generated by real functions (defined on $(0,\infty)$) which are concave around the point $1$ and possess certain upper estimates near $0$ and $\infty$. It turns out that their concave envelopes can be completely determined. Using this description, we establish sufficient conditions for the Hardy property of the homogeneous quasideviation mean and we also furnish an upper estimates for its Hardy constant.
\end{abstract}
\maketitle

\section{Introduction}

The origin of Hardy property goes back to the paper \cite{Har20a}, where this property was proved for all power means below the arithmetic means (with nonoptimal constant). Later this result was improved and extended by Landau \cite{Lan21}, Knopp \cite{Kno28}, and Carleman \cite{Car32} whose results are summarized in the inequality
\Eq{*}{
  \sum_{n=1}^\infty \P_p(x_1,\dots,x_n) \le C(p) \sum_{n=1}^\infty x_n
}
which holds for every sequences $(x_n)_{n=1}^\infty$ with positive terms, where $\P_p$ denotes the $p$-th \emph{power mean},
\Eq{*}{
C(p):=
\begin{cases} 
(1-p)^{-1/p}&p \in (-\infty,0) \cup (0,1), \\ 
e & p=0, \\
\infty & p\in[1,\infty),
\end{cases} 
}
and this constant is sharp, i.e., it cannot be diminished. For more details about the history of the developments related to Hardy type inequalities, see papers Pe\v{c}ari\'c--Stolarsky \cite{PecSto01}, Duncan--McGregor \cite{DunMcg03}, and the book of Kufner--Maligranda--Persson \cite{KufMalPer07}.

It has been also extensively studied by the authors since 2016. There are a number results which allows one to obtain the Hardy constant for the mean. All these results use Kedlaya (or Kedlaya-type) properties \cite{Ked94,Ked99} in their background, which unifies their assumptions. In the most natural setting, we assume that a mean is concave, homogeneous, and
repetition invariant (then it is also monotone).These assumptions are relaxed for example using homogenizations techniques \cite{PalPas19a,PalPas20}, or by replacing repetition invariance by a weaker axiom \cite{Pas21a}. However we have not been able to relax the concavity assumption. Due to this reason, it is difficult to establish the Hardy constant for means which are nonconcave. 

The most recent idea, which arises from the paper \cite{PalPas23a} is to majorize the mean (in the mentioned paper it was a Gini mean) by certain concave mean and calculate its Hardy constant, which is an upper estimate of the Hardy constant of the initial one. The aim of this paper is to genealize this idea to a broader class of means.

Remarkably, as it was mentioned in the Introduction to the paper \cite{PalPas23a}, Gini means was considered in the quasideviation framework. This was the most natural setting due to the earlier, comprehensive study of the Hardy property in this family \cite{PalPas20}. Indeed, the key tool in was to use \cite[Lemma~3.2]{PalPas23a}, which in fact reduces the problem of finding the upper estimate for nonconcave Gini means to the problem of finding the Hardy constant of the corresponding quasideviation means which were already homogeneous and concave.

In this paper we focus on homogeneous quasideviation means generated by functions belonging to a certain class $\Phi$, which means that they are concave around the point $1$ and satisfy some further assumptions. It turns out that their concave envelopes are of a special form which will be described in Theorem~\ref{thm:1}. Using this description, in Theorem~\ref{thm:main1}, we obtain sufficient conditions for the Hardy property of such homogeneous quasideviation means and then, in Theorem~\ref{thm:main2}, we establish upper estimates for their Hardy constants. 

\section{Means and their properties}

A function $\M \colon \bigcup_{n=1}^\infty \R_+^n \to \R_+$ such that $\min(x)\le \M(x)\le \max(x)$ holds for all $x$ in the domain of $\M$ is called a \emph{mean (on $\R_+$)}.Throughout the present note all considered means are defined on $\R_+$, thus we can omit the domain of a mean whenever convenient. We also adopt the standard convention that natural properties like convexity, homogeneity, etc. refer to the respective properties of the $n$-variable function $\M|_{\R_+^n}$ to be valid for all $n \in\N$. 

For a given mean $\M$ let $\Hc(\M)$ denote the smallest nonnegative extended real
number, called the \emph{Hardy constant of $\M$}, such that, for all sequences $(x_1,x_2,\dots)$ of positive elements,
\Eq{*}{
\sum_{n=1}^\infty \M(x_1,\dots,x_n) \le \Hc(\M) \cdot \sum_{n=1}^\infty x_n.
}
Means with a finite Hardy constant are called \emph{Hardy means} (cf.\ \cite{PalPas16}). There are a number of general results for this property obtained by the authors, however we would like to omit them and proceed to the setup of homogeneous quasideviation means. 


\subsection{Homogeneous quasideviation means}

In what follows, we recall the notions of a quasideviation mean (cf.\ \cite{Pal82a}, \cite{Pal89b}).

\begin{defin} A function $E\colon I\times I\to\R$ is said to be a \emph{quasideviation} if
\begin{enumerate}[(a)]
 \item for all elements $x,y\in I$, the sign of $E(x,y)$ coincides with that of $x-y$,
 \item for all $x\in I$, the map $I\ni y\mapsto E(x,y)$ is continuous and,
 \item for all $x<y$ in $I$, the mapping $(x,y)\ni t\mapsto \frac{E(y,t)}{E(x,t)}$ is strictly increasing.
\end{enumerate}
By the results of the paper \cite{Pal82a},
for all $n\in\N$ and $x=(x_1,\dots,x_n)\in I^n$, the equation
\Eq{e}{
  E(x_1,y)+\cdots+E(x_n,y)=0
}
has a unique solution $y$, which will be called the \emph{$E$-quasideviation mean} of $x$ and denoted by $\D_E(x)$. One can easily notice that power means, quasiarithmetic means, Gini means are quasideviation means (see \cite{Pal82a}, \cite{Pal89b} for further details).
\end{defin}

Let $\F$ denote a class of all continuous functions $f\colon\R_+\to\R$ such that 
\begin{enumerate}[(i)]
 \item $\sign(f(t))=\sign(t-1)$ for all $t\in\R_+$, 
 \item for all $x\in (0,1)$, the mapping
$t \mapsto \frac{f(t)}{f(t/x)}$ is strictly increasing on $(x,1)$. 
\end{enumerate}

\begin{lem}[\cite{PalPas23a}, Lemma~2.1] \label{L:HQD} Let $f:\R_+\to\R$ be a function such that $\sign(f(x))=\sign(x-1)$ holds for all $x\in\R_+$ and assume that, for some $p\in\R$, the function $f_p(t):=t^pf(t)$ ($t\in\R_+$) is increasing on $\R_+$ and strictly increasing on $(0,1)$. Then $f$ belongs to $\F$. In particular, if $f$ is increasing on $\R_+$ and strictly increasing on $(0,1)$, then $f\in\F$.
\end{lem}

By \cite{Pal88e}, we know that for all $f \in \F$ the functions $E_f \colon \R_+^2 \to \R$ given by $E_f(x,y):=f(\frac xy)$ is a quasideviation and the corresponding quasideviation mean $\E_f:=\D_{E_f}$ is homogeneous and continuous. Conversely, if $\D_E$ is homogeneous and continuous, then $\D_E=\E_f$ for some $f \in \F$. One can easily check that  power means and Gini means are homogeneous and continuous.

The following results of papers \cite{PalPas18a,PalPas20} are instrumental for us.

\begin{lem}[\cite{PalPas18a}, Theorem~2.3]\label{lem:PalPas18a}
Let $f\colon\R_+\to\R$ be concave such that $\sign(f(x))=\sign(x-1)$ for all $x\in\R_+$. Then $f\in\F$, the function $E_f\colon\R_+^2\to\R$ defined by $E_f(x,y):=f\big(\frac xy\big)$ is a quasideviation and the quasideviation mean $\E_f:=\D_{E_f}$ is homogeneous, continuous, nondecreasing and concave.
\end{lem}

\begin{lem}[\cite{PalPas20}, Theorem~5.4]\label{lem:PalPas20}
Let $f:\R_+\to\R$ be a concave function such that $\sign(f(x))=\sign(x-1)$ holds for all $x\in\R_+$. Then the homogeneous quasideviation mean $\E_f$ is a Hardy mean if and only if the function $x\mapsto f\big(\frac1x\big)$ is integrable over $(0,1]$. In the latter case, $c:=\Hc(\E_f)$ is the unique solution of the equation 
 \Eq{*}{
 \int_0^c f\big(\tfrac1x\big)\:dx=0.
 }.
\end{lem}

Finally, let us recall the following result about the comparison of homogeneous quasideviation means. (See, for example \cite{DarPal82}, \cite[Theorem~10]{Pal88a}, and \cite[Lemma~2.3]{PalPas23a}.)
\begin{lem}\label{lem:comparizon}
For all $f,g \in \F$ with $f \le g$, we have $\E_f\le \E_g$. 
\end{lem}

\section{The concave envelope of functions belonging to the class $\Phi$}

In this section, we introduce a subclass of homogeneous quasideviation means which we are going to discuss. Let $\FAM$ denote the family of all continuous functions $g \colon \R_+ \to \R$ satisfying the following properties 
\begin{enumerate}[($\Phi1$)]
\item $\sign g(t)=\sign(t-1)$, 
\item $g(t) \le t-1$ for all $t \in \R_+$, 
\item there exist $\alpha,\beta\in [0,\infty]$ with $\alpha<1<\beta$ such that
 \begin{enumerate}[($\Phi3a$)]
  \item $g$ in concave on $(\alpha,\beta)$,
  \item $g$ possesses the following convexity type condition on $(0,\alpha]$:
  \Eq{*}{
    g(t)\leq \tfrac{t}{\alpha}g(\alpha)
    +\tfrac{\alpha-t}{\alpha}g_+(0)
    \qquad(t\in(0,\alpha]),
  }
  where $g_+(0):=\limsup_{s\to0}g(s)$.
  \item $g$ possesses the following convexity type condition on $[\beta,\infty)$,
  \Eq{*}{
    g(t)\leq g(\beta)+q(t-\beta) \qquad(t\in[\beta,\infty)),
  }
  where $q:=\limsup_{s\to\infty}g(s)/s$.
 \end{enumerate}
\end{enumerate}

\subsection{The concave envelope} In this section we aim to describe the concave envelope of functions belonging to the class $\Phi$. 
To do this, we first recall the definition of the concave envelope of real-valued functions from the paper \cite{PalPas21a}. Namely, for a function $f \colon I \to \R$, which possesses a concave majorization, we define $\conc(f) \colon I \to \R$ by 
\Eq{*}{
\conc(f)(t)&:=\inf \{g(t) \mid g \colon I \to \R,\,\, g \text{ is a concave function, and }g \ge f \}.
}
In view of the results of the mentioned paper, we know that the concave envelope admits some natural properties like continuity, concavity and the inequality $f\leq\conc(f)$ holds. In general, for an arbitrary function $f\colon I\to\R$, it could happen that there is no concave function above $f$. On the other hand, each member of $\Phi$ is bounded from above by the affine function $t\mapsto t-1$, thus the concave envelope exists for each element of $\Phi$.

In order to describe the concave envelope of functions belonging to $\Phi$, we introduce the following general transformation. Given an open interval $I\subseteq\R_+$, $a,b \in I\cup\{\inf I,\sup I\}$ with $a<b$ and $p,q \in\R$, for any function $g \colon I \to \R$, we define $\Gamma_{a,b;p,q}(g) \colon \R_+ \to \R$ by
\Eq{*}{
\Gamma_{a,b;p,q}(g)(t):=\begin{cases}
        g(a)+p(t-a) & \text{ for }t \in (0,a],\\
        g(t) & \text{ for }t \in (a,b),\\
        g(b)+q(t-b) & \text{ for }t \in [b,+\infty).
             \end{cases}
}
It is easy to check that if $g$ is continuous then all functions of the form above are also continuous. It is woth noticing that $\Gamma_{a,b;p,q}(g)$ does not depend on $p$ and $q$ if $a=0$ and $b=+\infty$, respectively. In our first main theorem, we show the concave envelope of the elements of $\Phi$ are of this form.
 
For to formulation of the subsequent results, we recall the notions of \emph{left lower} and the \emph{right upper Dini derivatives}: For a function $g:I\to\R$ and a point $x\in I$, they are defined by
\Eq{*}{
  D_-g(x):=\liminf_{t\uparrow x}\frac{g(t)-g(x)}{t-x}
  \qquad\mbox{and}\qquad
  D^+g(x):=\limsup_{t\downarrow x}\frac{g(t)-g(x)}{t-x},
}
respectively.

\begin{thm}\label{thm:1}
Let $g \in \Phi$ and let $\alpha,\beta$ be chosen according to property ($\Phi3$). Then there exist $a \in [\alpha,1]$ and $b \in [1,\beta]$ such that $\conc(g)=\Gamma_{a,b;p,q}(g)$, where
\Eq{*}{
  p:=\begin{cases}
     \liminf\limits_{t \to 0^+} \dfrac{g(a)-g(t)}{a-t} & \mbox{if } \alpha>0,\\[2mm]
     1 & \mbox{if } \alpha=0
     \end{cases}
  \qquad\text{and}\qquad 
  q:=\begin{cases}
     \limsup\limits_{t \to \infty} \dfrac{g(t)}{t} & \mbox{if } \beta<\infty,\\[2mm]
     1 & \mbox{if } \beta=\infty.
     \end{cases}
}
Furthermore $0\leq q\leq 1$, in addition, if $0<\alpha$ and $\beta<+\infty$ then they also satisfy
\Eq{E:gagb}
{
D^+g(a)\leq p\leq D_-g(a)
\qquad\text{and}\qquad 
D^+g(b)\leq q\leq D_-g(b),
}
respectively.
\end{thm}

\begin{proof} If $\alpha=0$, then define $a:=0$ and $p:=1$.
Similarly, if $\beta=\infty$, then define $b:=\infty$ and $q:=1$.

Assume now that $\alpha>0$. Due to the property ($\Phi2$), we can see that $g$ has a finite upper right limit at $0$, moreover,
\Eq{g+}{
  g_+(0)\leq \limsup_{t\to0}(t-1)=-1.
}
Define the map $\varphi:\R_+\to\R$ by
\Eq{*}{
  \varphi(t):=\frac{g(t)-g_+(0)}{t}.
}
Using condition $(\Phi3b$), for all $t\in(0,\alpha]$, we get that
\Eq{*}{
   \varphi(t)\leq \frac{\big(\frac{t}{\alpha}g(\alpha)
    +\frac{\alpha-t}{\alpha}g_+(0)\big)-g_+(0)}{t}
    =\frac{g(\alpha)-g_+(0)}{\alpha}
    =\varphi(\alpha).
}
On the other hand, if $t\in[1,\infty)$, then using property $(\Phi2)$ and $1+g_+(0)\leq 0$, it follows that
\Eq{*}{
  \varphi(t)\leq\frac{t-1-g_+(0)}{t}
  =1-\frac{1+g_+(0)}{t}\leq -g_+(0)=\varphi(1).
}
Therefore, in view of its continuity, the function $\varphi$ is bounded from above and its maximum is attained at an element $a\in[\alpha,1]$. Thus, for all $t\in\R_+$, we have that
\Eq{*}{
  \varphi(t)\leq\varphi(a)
  =\frac{g(a)-\limsup_{s\to0}g(s)}{a}
  =\liminf_{s\to0}\frac{g(a)-g(s)}{a-s}=:p.
}
This inequality, for all $t\in\R_+$, implies that
\Eq{*}{
  g(t)\leq \tfrac{t}{a}g(a)+\tfrac{a-t}{a}g_+(0).
}
Therefore, 
\Eq{*}{
  D^+g(a)&\leq \limsup_{t\downarrow a}\frac{\big(\frac{t}{a}g(a)+\frac{a-t}{a}g_+(0)\big)-g(a)}{t-a}
  =\frac{g(a)-g_+(0)}{a}=p,\\
  D_-g(a)&\geq \liminf_{t\uparrow a}\frac{\big(\frac{t}{a}g(a)+\frac{a-t}{a}g_+(0)\big)-g(a)}{t-a}
  =\frac{g(a)-g_+(0)}{a}=p,
}
which shows the validity of the first two inequalities in \eq{E:gagb}.

If $\beta<\infty$, then define
\Eq{*}{
  q:=\limsup_{t\to\infty}\frac{g(t)}{t}.
}
Then, using properties $(\Phi1)$ and $(\Phi2)$, it follows that
\Eq{*}{
  0\leq q\leq \limsup_{t\to\infty}\frac{t-1}{t}=1,
}
hence $q\in[0,1]$. Now define the function $\psi:\R_+\to\R$ by $\psi(t):=g(t)-qt$. 

In view of condition $(\Phi3c)$, for all $t\in[\beta,\infty)$, we get that
\Eq{*}{
  \psi(t)\leq\psi(\beta).
}
Additionally, applying conditions $(\Phi2)$ and $(\Phi1)$, for $t\in(0,1]$, we get that
\Eq{*}{
  \psi(t)\leq (t-1)-qt=(1-q)t-1\leq -q=\psi(1).
}
This, using also the continuity of $\psi$, shows that $\psi$ is bounded from above over $\R_+$ and it attains its maximum at an element $b\in[1,\beta]$. Consequently, for all $t\in\R_+$, we have that
\Eq{*}{
  g(t)-qt\leq g(b)-qb.
}
Using this inequality, we get
\Eq{*}{
  D^+g(b)&\leq \limsup_{t\downarrow b}\frac{\big(g(b)+q(t-b)\big)-g(b)}{t-b}=q,\\
  D_-g(b)&\geq \liminf_{t\uparrow a}\frac{\big(g(b)+q(t-b)\big)-g(b)}{t-b}=q,
}
which shows the validity of the last two inequalities in \eq{E:gagb}.

To complete the proof, we need to show that $\conc(g)=\Gamma_{a,b;p,q}(g)$. 

Since $g$ is concave on $(\alpha,\beta)$, it has left and right derivatives at each point of this interval.
The function $h:=\Gamma_{a,b;p,q}(g)$ is continuous on $\R_+$ and for its right and left derivatives (using just its definition), we have the following formulas
\Eq{*}{
  D^+h(t)
  =\begin{cases}
    p & \mbox{if } t\in(0,a),\\
    D^+g(t) & \mbox{if } t\in[a,b),\\
    q & \mbox{if } t\in[b,\infty),
   \end{cases}
   \qquad
  D_-h(t)
  =\begin{cases}
    p & \mbox{if } t\in(0,a],\\
    D_-g(t) & \mbox{if } t\in (a,b],\\
    q & \mbox{if } t\in(b,\infty).
   \end{cases}
}
Using the concavity of $g$ over $[a,b]\subseteq[\alpha,\beta]$ we have that $D^+g$ and $D_-g$ are decreasing functions over $[a,b)$ and $(a,b]$, respectively, therefore, in view of the inequalities in \eq{E:gagb}, it follows that 
\Eq{*}{
&p\geq D^+g(a)\geq D^+g(t)\geq D^+g(s)\geq D_-g(b)\geq q \qquad(t,s\in[a,b),\, t\leq s)
\qquad\mbox{and}\\
&p\geq D^+g(a)\geq D_-g(t)\geq D_-g(s)\geq D_-g(b)\geq q \qquad (t,s\in(a,b],\, t\leq s).
}
Thus, we have established that the right (as well as the left) derivative of $h$ is a decreasing function over $\R_+$, which proves that $h=\Gamma_{a,b;p,q}(g)$ is concave over $\R_+$.

\comment{Using that $a\leq1$ and \eq{E:gagb}, we can see that
\Eq{*}{
p \geq D^+g(a)\geq D^+g(1).
}
Thus $1 \le p$.}

We show now that $\Gamma_{a,b;p,q}(g)(t)\geq g(t)$ for all $t\in\R_+$. This inequality is in fact an equality if $t\in(a,b)$. Assume that $t\in(0,a]$. Then, by the inequality $\varphi(t)\leq\varphi(a)$, we have that
\Eq{*}{
  g(t)\leq g_+(0)+t\frac{g(a)-g_+(0)}{a}
  =g_+(0)+tp=g(a)+p(t-a).
}
This proves that $\Gamma_{a,b;p,q}(g)(t)\geq g(t)$ holds for all $t\in(0,a]$.

If $t\in[b,\infty)$, then, by the inequality $\psi(t)\leq\psi(b)$, we have that
\Eq{*}{
  g(t)\leq g(b)+q(t-b),
}
which shows that $\Gamma_{a,b;p,q}(g)(t)\geq g(t)$ is valid for all $t\in[b,\infty)$.

Finally, we verify that $\Gamma_{a,b;p,q}(g)$ is minimal among those concave functions that are nonsmaller than $g$ on $\R_+$. Let $f:\R_+\to\R$ be a concave function such that $f(t)\geq g(t)$ for all $t\in\R_+$. The inequality $f(t)\geq \Gamma_{a,b;p,q}(g)(t)$ is obvious for $t\in(a,b)$. 

Due to the inequality $g\leq f$ and the concavity of $f$ over $\R_+$, we can see that the right limit of $f$ exists at $0$ and is nonsmaller than $g_+(0)$. Thus we can extend $h$ continuously to $[0,\infty)$ and this extension will also be concave. Therefore, for $t\in(0,a]$,
\Eq{*}{
  f(t)=f\big(\tfrac{a-t}{a}\cdot 0+\tfrac{t}{a}\cdot a\big)
  &\geq \tfrac{a-t}{a}f_+(0)+\tfrac{t}{a}f(a)
  \geq \tfrac{a-t}{a}g_+(0)+\tfrac{t}{a}g(a)\\
  &=g(a)+p(t-a)=\Gamma_{a,b;p,q}(g)(t).
}
For $t\in[b,\infty)$, let $s>t$ be arbitrary. Then, by the concavity of $f$, we get
\Eq{*}{
  f(t)=f\big(\tfrac{s-t}{s-b}b+\tfrac{t-b}{s-b}s\big)
  \geq \tfrac{s-t}{s-b}f(b)+\tfrac{t-b}{s-b}f(s)
  \geq \tfrac{s-t}{s-b}g(b)+\tfrac{t-b}{s-b}g(s)
}
We can see that
\Eq{*}{
  \limsup_{s\to\infty}\frac{g(s)}{s-b}
  =\limsup_{s\to\infty}\frac{g(s)/s}{1-(b/s)}=q.
}
Using this equality, upon taking the limsup as $s\to\infty$ in the previous inequality, it follows that
\Eq{*}{
  h(t)\geq 
  \limsup_{s\to\infty}\Big(\tfrac{s-t}{s-b}g(b)+\tfrac{t-b}{s-b}g(s)\Big)
  =g(b)+q(t-b)=\Gamma_{a,b;p,q}(g)(t).
}
This completes the proof of the equality $\conc(g)=\Gamma_{a,b;p,q}(g)$.
\end{proof}

\section{The Hardy property and the Hardy constant}

Once we already know the form of concave envelopes of function is $\Phi$, we can use them to establish upper bounds of homogeneous quasideviations means which are generated by functions belonging to this family.

\begin{lem}
 \label{cor:2}
For every function $g \in \mathcal{F}$ which is bounded from above and satisfies the inequality $g(t)\le t-1$ for all $t \in \R_+$, the mean $\E_g$ possesses the Hardy property and its Hardy constant is less than or equal to $c$, where $c>1$ is the unique solution of the equation
\Eq{c}{
  c-1-\ln(c)=\ln\big(\sup\nolimits_{\R_+} g+1\big).
}
In particular, for all $n\geq2$,
\Eq{cc}{
  c\leq \exp\bigg(\sqrt[n]{n!\ln\big(\sup\nolimits_{\R_+} g+1\big)}\bigg).
}
\end{lem}

\begin{proof}
 Let $M:=\sup_{\R_+} g>0$ and define $h \colon \R_+ \to \R$ by 
 \Eq{h}{
h(t):=\min(t-1,M)=
\begin{cases}
 t-1 & \text{for }t \le M+1,\\
 M & \text{for }t> M+1.
\end{cases}
}
Then we have that $h\in\F$ and $g(t) \le h(t)$ for all $t \in \R_+$. Therefore $\E_g\le \E_h$ and it is sufficient to show that $\E_h$ is a Hardy mean. However, $h$ is obviously concave and 
$\sign(h(t))=\sign(t-1)$ holds for all $t \in \R_+$. Then, by Lemma~\ref{lem:PalPas20}, $\E_h$ is a Hardy mean if and only if the function $x \mapsto h(1/x)$ is integrable over $(0,1]$. 
Indeed, we have that
\Eq{*}{
\int_0^1 h\big(\tfrac1x\big)dx&=\int_0^{\frac{1}{M+1}} h\big(\tfrac1x\big)dx+\int_{\frac{1}{M+1}}^1h\big(\tfrac1x\big)dx
=\int_0^{\frac{1}{M+1}} M\:dx+\int_{\frac{1}{M+1}}^1 \big(\tfrac{1}{x}-1\big)\:dx\\
&=\tfrac{M}{M+1}+\ln(M+1)-\tfrac{M}{M+1}=\ln(M+1)<+\infty,
}
and we are done.

The Hardy constant $c=\Hc(\E_h)$ is the solution of the equation
\Eq{*}{
0=\int_0^c h\big(\tfrac1x\big)dx=\ln(c(M+1))+1-c,
}
which shows that \eq{c} is valid. To verify the last inequality, we use that $c>1$ and estimate
\Eq{*}{
  c-1-\ln(c)&=\exp(\ln(c))-1-\ln(c)
  =\sum_{k=0}^\infty \frac{(\ln(c))^k}{k!}-1-\ln(c)\\
  &=\sum_{k=2}^\infty \frac{(\ln(c))^k}{k!}
  \geq \frac{(\ln(c))^n}{n!} \qquad(n\geq2).
}
This inequality then directly implies that \eq{cc} is valid.
\end{proof}

\begin{rem}
One can easily compute the $\E_h$ mean of $0<x_1\leq x_2\leq\dots\leq x_n$, where $h$ is given by \eq{h} and $M$ is a positive number. Then
$y:=\E_h(x_1,\dots,x_n)$ is the solution of the equation
\Eq{*}{
  \sum_{i=1}^k\Big(\frac{x_i}{y}-1\Big)+\sum_{i=k+1}^n M=0
}
and $k\in\{1,\dots,n\}$ is the largest index such that $x_k\leq y(M+1)$. From this equality, we get that
\Eq{y}{
  y=\frac{1}{k(M+1)-nM}\sum_{i=1}^k x_i.
}
The condition $x_k\leq y(M+1)$ is now equivalent to the inequality
\Eq{k}{
  x_k\leq \frac{1}{k-n\frac{M}{M+1}}\sum_{i=1}^k x_i.
}
Therefore, to compute $y=\E_h(x_1,\dots,x_n)$, one has to find the largest index $k\in\{1,\dots,n\}$ such that \eq{k} (and hence $k>n\frac{M}{M+1}$) be valid, then $y$ is given by \eq{y}.
\end{rem}

In the following result, for any $g \in \F \cap \Phi$, we charaterize the Hardy property of the mean $\E_{\conc(g)}$ and, consequently, we establish a sufficient condition for the Hardy property of $\E_{g}$.

\begin{thm}\label{thm:main1} Let $g \in \F \cap \Phi$ and let $\alpha\in[0,1)$ and $\beta\in(1,+\infty]$ according to property $(\Phi3)$ of $g$. Then $\conc(g)\in\F$ and the mean $\E_{\conc(g)}$ possesses the Hardy property if and only if either $\beta=+\infty$ and $\int_0^1 g(\frac{1}{t})dt<+\infty$, or $\beta<+\infty$ and $g$ is bounded from above. In these cases, we have that $\E_g$ is also a Hardy mean and $\Hc(\E_g)\le \Hc(\E_{\conc(g)})$.
\end{thm}

\begin{proof} In view of Theorem~\ref{thm:1}, the concave envelope of $g$ is nondecreasing, which implies that it belongs to $\F$.

According to Lemma \ref{lem:PalPas20}, $\E_{\conc(g)}$ is a Hardy mean if and only if the integral $\int_0^1 \conc(g)(\frac{1}{t})dt$ is finite. To check this property, we distinguish two cases. 

If $\beta=+\infty$, then $\conc(g)(t)=g(t)$ for all $t\geq1$.
Consequently, $\int_0^1 \conc(g)(\frac{1}{t})dt=\int_0^1 g(\frac{1}{t})dt$. Thus, $\int_0^1 \conc(g)(\frac{1}{t})dt$ is finite if and only if $\int_0^1 g(\frac{1}{t})dt$ is finite.

If $\beta<+\infty$, then there exists $b\in[1,\beta]$ and $q\ge0$ such that, for all $t\in[b,\infty)$,
\Eq{q}{
  \conc(g)(t)=g(b)+q(t-b).
}
The integral $\int_0^1 \conc(g)(\frac{1}{t})dt$ is finite if and only if $\int_0^{\frac{1}{b}} \conc(g)(\frac{1}{t})dt$ is finite. Then, in view of the equality \eq{q}, we have that
\Eq{*}{
  \int_0^{\frac{1}{b}} \conc(g)(\tfrac{1}{t})dt
  =\int_0^{\frac{1}{b}} g(b)+q(\tfrac{1}{t}-b)dt,
}
which can be finite if and only if $q=0$. In this case, $g$ is bounded from above.

On the other hand, if $g$ is bounded from above, then, according to the Lemma~\ref{cor:2}, it follows that $\E_{\conc(g)}$ is a Hardy mean.

In view of the inequality $g\leq \conc(g)$ and Lemma~\ref{lem:comparizon}, it follows that $\E_g\leq\E_{\conc(g)}$. Therefore, $\E_g$ is also a Hardy mean and $\Hc(\E_g)\le \Hc(\E_{\conc(g)})$ holds.
\end{proof}

In what follows, we describe the Hardy constant of the mean $\E_{\conc(g)}$ provided that it is finite.

\begin{thm}\label{thm:main2}
Let $g \in \F \cap \Phi$ and let $\alpha\in[0,1)$ and $\beta\in(1,+\infty]$ according to property $(\Phi3)$ of $g$. Assume that $\E_{\conc(g)}$ is a Hardy mean and choose $a\in[\alpha,1]$, $b\in[1,\beta]$ and $p,q\in\R$ as stated in Theorem~\ref{thm:1}. 
Then the Hardy constant $c>1$ of the mean $\E_{\conc(g)}$ is determined by the one of following conditions:
\begin{enumerate}[(i)]
 \item If $\alpha=0$ and $\beta=\infty$, then $c$ is the unique solution of the equality
 \Eq{cg+}{
   \int_0^c g(\tfrac{1}{t})dt=0.
 }
 \item If $\alpha=0$ and $\beta<\infty$, then $c$ is the unique solution of the equality
 \Eq{*}{
 \frac{g(b)}{b}+\int_{1/b}^{c}g(\tfrac1t)\:dt=0.
}
 \item If $\alpha>0$, then $c$ is the unique solution of the equality
 \Eq{E:Kg}{
K(g)-\int_{c}^{1/a}g(\tfrac1t)dt&=0 & \text{ if }K(g)\le 0,\\
K(g)+p \ln(ca)+(c-\tfrac1a)(g(a)-pa) &= 0 & \text{ if }K(g)>0,
}
where, for $\beta=+\infty$,
\Eq{*}{
K(g):=
    \int_{0}^{1/a} g(\tfrac1t)\:dt
    =\int_{a}^{\infty} \frac{g(s)}{s^2}\:ds
}
and, for $\beta<+\infty$,
\Eq{*}{
  K(g):=\frac{g(b)}b+\int_{1/b}^{1/a} g(\tfrac1t)\:dt
    =\frac{g(b)}b+\int_{a}^{b} \frac{g(s)}{s^2}\:ds.
}
\end{enumerate}
\end{thm}

\begin{proof} According to Lemma~\ref{lem:PalPas20}, the Hardy constant $c$ of the mean $\E_{\conc(g)}$ is the unique solution of the equation $\psi(c)=0$, where $\psi:\R+\to\R$ is defined by
 \Eq{ceg}{
   \psi(x):=\int_0^x \conc(g)(\tfrac{1}{t})dt.
 }
We rewrite the equation $\psi(c)=0$ in each of the cases listed in the theorem. Using Theorem~\ref{thm:1}, we also have that $\conc(g)=\Gamma_{a,b;p,q}(g)$.

If $\alpha=0$ and $\beta=\infty$, then $g=\conc(g)$, therefore, $\psi(c)=0$ is equivalent to \eq{cg+}.

If $\alpha=0$ and $\beta<\infty$, then $\E_{\conc(g)}$ is a Hardy mean if and only if $q=0$. Now $a=0$ and $\frac1b\leq1<c$, thus
\Eq{*}{
0=\psi(c)=\int_0^{1/b} \conc(g)(\tfrac1t)\:dt+\int_{1/b}^{c} \conc(g)(\tfrac1t)\:dt
=\frac{g(b)}{b}+\int_{1/b}^{c}g(\tfrac1t)\:dt.
}

If $\alpha>0$, then we distinguish two subcases.

(A). If $c \in (1,\tfrac1a]$, then $\psi(\tfrac1a)\le \psi(c)=0$ and 
\Eq{*}{
0=\psi(c)=\int_0^{c} \conc(g)(\tfrac1t)\:dt
=\psi(\tfrac1a)-\int_{c}^{1/a} \conc(g)(\tfrac1t)\:dt=\psi(\tfrac1a)-\int_{c}^{1/a}g(\tfrac1t)dt.
}

(B). If $c \in (\tfrac{1}a,\infty)$, then $\psi(\tfrac1a)> \psi(c)=0$. Then
\Eq{*}{
\psi(c)&=\int_0^{c} \conc(g)(\tfrac1t)\:dt
=\psi(\tfrac1a)+\int_{1/a}^{c} \conc(g)(\tfrac1t)\:dt\\
&=\psi(\tfrac1a)+\int_{1/a}^{c} p(\tfrac1t-a)+g(a)\:dt
=\psi(\tfrac1a)+p \ln(ca)+(c-\tfrac1a)(g(a)-pa).
}
Finally, observe that 
\Eq{*}{
\psi(\tfrac1a)
=\int_{0}^{1/a} \conc(g)(\tfrac1t)\:dt=\int_{0}^{1/a} g(\tfrac1t)\:dt=K(g),
}
if $\beta=+\infty$ and
\Eq{*}{
\psi(\tfrac1a)
=\int_{0}^{1/b} g(b)\:dt+\int_{1/b}^{1/a} g(\tfrac1t)\:dt
=\frac{g(b)}b+\int_{1/b}^{1/a} g(\tfrac1t)\:dt=K(g),
}
if $\beta<+\infty$.

This completes the proof of the assertion in all of the cases.
\end{proof}


\begin{thebibliography}{10}

\bibitem{Car32}
T.~Carleman.
\newblock {Sur les fonctions quasi-analitiques}.
\newblock {\em Conférences faites au cinquième congrès des mathématiciens
  scandinaves, Helsinki}, page 181–196, 1932.

\bibitem{DarPal82}
Z.~Daróczy and Zs. Páles.
\newblock {On comparison of mean values}.
\newblock {\em Publ. Math. Debrecen}, 29(1-2):107–115, 1982.

\bibitem{DunMcg03}
J.~Duncan and C.~M. McGregor.
\newblock {Carleman's inequality}.
\newblock {\em Amer. Math. Monthly}, 110(5):424–431, 2003.

\bibitem{Har20a}
G.~H. Hardy.
\newblock {Note on a theorem of Hilbert.}
\newblock {\em {Math. Z.}}, 6:314–317, 1920.

\bibitem{Ked94}
K.~S. Kedlaya.
\newblock {Proof of a mixed arithmetic-mean, geometric-mean inequality}.
\newblock {\em Amer. Math. Monthly}, 101(4):355–357, 1994.

\bibitem{Ked99}
K.~S. Kedlaya.
\newblock {Notes: {A} {W}eighted {M}ixed-{M}ean {I}nequality}.
\newblock {\em Amer. Math. Monthly}, 106(4):355–358, 1999.

\bibitem{Kno28}
K.~Knopp.
\newblock {Über {R}eihen mit positiven {G}liedern}.
\newblock {\em J. London Math. Soc.}, 3:205–211, 1928.

\bibitem{KufMalPer07}
A.~Kufner, L.~Maligranda, and L.E. Persson.
\newblock {\em {The Hardy Inequality: About Its History and Some Related
  Results}}.
\newblock Vydavatelskỳ servis, 2007.

\bibitem{Lan21}
E.~Landau.
\newblock {A note on a theorem concerning series of positive terms}.
\newblock {\em J. London Math. Soc.}, 1:38–39, 1921.

\bibitem{Pas21a}
P.~Pasteczka.
\newblock {Online premeans and their computation complexity}.
\newblock {\em Results Math.}, 76(3):Paper No. 141, 23, 2021.

\bibitem{PecSto01}
J.~E. Pečarić and K.~B. Stolarsky.
\newblock {Carleman's inequality: history and new generalizations}.
\newblock {\em Aequationes Math.}, 61(1–2):49–62, 2001.

\bibitem{Pal82a}
Zs. Páles.
\newblock {Characterization of quasideviation means}.
\newblock {\em Acta Math. Acad. Sci. Hungar.}, 40(3-4):243–260, 1982.

\bibitem{Pal88a}
Zs. Páles.
\newblock {General inequalities for quasideviation means}.
\newblock {\em Aequationes Math.}, 36(1):32–56, 1988.

\bibitem{Pal88e}
Zs. Páles.
\newblock {On homogeneous quasideviation means}.
\newblock {\em Aequationes Math.}, 36(2-3):132–152, 1988.

\bibitem{Pal89b}
Zs. Páles.
\newblock {A {H}ahn-{B}anach theorem for separation of semigroups and its
  applications}.
\newblock {\em Aequationes Math.}, 37(2-3):141–161, 1989.

\bibitem{PalPas16}
Zs. Páles and P.~Pasteczka.
\newblock {Characterization of the {H}ardy property of means and the best
  {H}ardy constants}.
\newblock {\em Math. Inequal. Appl.}, 19(4):1141–1158, 2016.

\bibitem{PalPas18a}
Zs. Páles and P.~Pasteczka.
\newblock {On the best {H}ardy constant for quasi-arithmetic means and
  homogeneous deviation means}.
\newblock {\em Math. Inequal. Appl.}, 21(2):585–599, 2018.

\bibitem{PalPas19a}
Zs. Páles and P.~Pasteczka.
\newblock {On the homogenization of means}.
\newblock {\em Acta Math. Hungar.}, 159(2):537–562, 2019.

\bibitem{PalPas20}
Zs. Páles and P.~Pasteczka.
\newblock {On {H}ardy type inequalities for weighted quasideviation means}.
\newblock {\em Math. Inequal. Appl.}, 23(3):971–990, 2020.

\bibitem{PalPas21a}
Zs. Páles and P.~Pasteczka.
\newblock {On the {J}ensen convex and {J}ensen concave envelopes of means}.
\newblock {\em Arch. Math.}, 116(4):423–432, 2021.

\bibitem{PalPas23a}
Zs. Páles and P.~Pasteczka.
\newblock {Estimating the {H}ardy constant of nonconcave {G}ini means}.
\newblock {\em Math. Inequal. Appl.}, 26, 2023.

\end{thebibliography}

\end{document}